\documentclass[11pt]{amsart}

\usepackage[T1]{fontenc}
\usepackage{lmodern}
\usepackage{microtype}
\usepackage{amsmath,amssymb,amsthm,mathtools}
\usepackage{geometry}
\usepackage{enumitem}
\usepackage{xcolor}
\usepackage{aliascnt}
\usepackage[hidelinks]{hyperref}
\hypersetup{
  pdftitle={Hyperelliptic Brill--Noether loci},
  pdfauthor={Nero Budur and An-Khuong Doan}
}
\usepackage[nameinlink,capitalise]{cleveref}

\allowdisplaybreaks
\setlist{itemsep=2pt,topsep=4pt}

\newtheorem{theorem}{Theorem}[section]

\newaliascnt{proposition}{theorem}
\newtheorem{proposition}[proposition]{Proposition}
\aliascntresetthe{proposition}

\newaliascnt{lemma}{theorem}
\newtheorem{lemma}[lemma]{Lemma}
\aliascntresetthe{lemma}

\newaliascnt{corollary}{theorem}
\newtheorem{corollary}[corollary]{Corollary}
\aliascntresetthe{corollary}

\newaliascnt{definition}{theorem}

\aliascntresetthe{definition}

\theoremstyle{remark}
\newaliascnt{remark}{theorem}
\newtheorem{remark}[remark]{Remark}
\aliascntresetthe{remark}

\newcommand{\PP}{\mathbf P}
\newcommand{\Aff}{\mathbf A}
\newcommand{\CC}{\mathbf C}
\newcommand{\QQ}{\mathbf Q}

\newcommand{\OO}{\mathcal O}
\newcommand{\Pic}{\operatorname{Pic}}
\newcommand{\Ext}{\operatorname{Ext}}
\newcommand{\Hom}{\operatorname{Hom}}
\newcommand{\Sym}{\operatorname{Sym}}
\newcommand{\Spec}{\operatorname{Spec}}
\newcommand{\rank}{\operatorname{rank}}
\newcommand{\Cat}{\operatorname{Cat}}

\newcommand{\IC}{\operatorname{IC}}
\newcommand{\MHM}{\operatorname{MHM}}

\newcommand{\id}{\operatorname{id}}

\newcommand{\Fitt}{\operatorname{Fitt}}
\newcommand{\RGamma}{\mathbf R\Gamma}
\newcommand{\Rp}{\mathbf R p_*}
\newcommand{\Rq}{\mathbf R q_*}
\newcommand{\Acal}{\mathcal A}
\newcommand{\Ecal}{\mathcal E}
\newcommand{\Fcal}{\mathcal F}
\newcommand{\Kcal}{\mathcal K}
\newcommand{\Lcal}{\mathcal L}
\newcommand{\Mcal}{\mathcal M}
\newcommand{\Ncal}{\mathcal N}
\newcommand{\Qcal}{\mathcal Q}

\newcommand{\QH}{\QQ^H}

\title[Hyperelliptic Brill--Noether loci]{Hyperelliptic Brill--Noether loci}

\author{Nero Budur}
\address{Department of Mathematics, KU Leuven, Celestijnenlaan 200B, 3001 Leuven, Belgium;   YMSC, Tsinghua University, 100084 Beijing, China}
\email{nero.budur@kuleuven.be}

\author{An-Khuong Doan}
\address{Institute for Advanced Studies in Mathematics, Harbin Institute of Technology, Harbin, China 150001}
\email{an-khuong.doan@hit.edu.cn}

\begin{document}

\begin{abstract}
Hyperelliptic theta divisors, and more generally,  Brill--Noether loci, are shown to be locally catalecticant determinantal varieties. This has many consequences for their local and global structure.
\end{abstract}

\maketitle

\section{Introduction}

Let $C$ be a smooth complex projective hyperelliptic curve of genus $g\geq2$.
For integers $d$ and $r\geq0$, the degree-$d$ Brill--Noether locus is
\[
 W^r_d(C)=\{L\in\Pic^d(C)\mid h^0(C,L)\geq r+1\},
\]
endowed with the
 scheme structure of \cite[Chapter IV, \S3]{ACGH}.  Without loss of generality, we work in
the  range
\[
 0\leq d\leq g-1.
\]
Fix $L\in W^0_d(C)$ and write
\begin{equation}\label{eq:basic-numerics}
 \ell=h^0(C,L)\geq1,\qquad
 \ell'=h^1(C,L)=g-d+\ell-1,\qquad
 b=d-2\ell+2.
\end{equation}
Since $d\leq g-1$,
\begin{equation}\label{eq:ellprime-ge-ell}
 \ell'-\ell=g-d-1\geq0.
\end{equation}
For  variables $z_0,\ldots,z_{\ell+\ell'-2}$, define the
$\ell'\times\ell$ rectangular  catalecticant matrix
\begin{equation}\label{eq:rectangular-cat}
 \Cat_{\ell',\ell}(z)=
 \bigl(z_{i+j}\bigr)_{
 0\leq i\leq\ell'-1,\ 0\leq j\leq\ell-1}
 =
 \begin{pmatrix}
 z_0&z_1&\cdots&z_{\ell-1}\\
 z_1&z_2&\cdots&z_\ell\\
 \vdots&\vdots&&\vdots\\
 z_{\ell'-1}&z_{\ell'}&\cdots&z_{\ell+\ell'-2}
 \end{pmatrix}.
\end{equation}
For $1\leq s\leq\ell$, denote by $I_s(\Cat_{\ell',\ell})$ the ideal of
its $s\times s$ minors.

\begin{theorem}\label{thm:main}
Let $C$, $d$, and $L$ be as above, and define $\ell$, $\ell'$, and $b$ by
\eqref{eq:basic-numerics}.  Then
$ b\geq0$ and
there is a  local $\CC$-isomorphism for the \'etale topology
\[
 (\Pic^d(C),L)\simeq_{\mathrm{\acute et}}
 (\Aff^{\ell+\ell'-1}\times\Aff^b,0)
\]
which, simultaneously for every $0\leq r\leq\ell-1$, induces an embedded
local \'etale isomorphism
\begin{equation}\label{eq:all-r-etale-normal-form}
 (\Pic^d(C),W^r_d(C),L)
 \simeq_{\mathrm{\acute et}}
 \bigl(\Aff^{\ell+\ell'-1}\times\Aff^b,
 V(I_{\ell-r}(\Cat_{\ell',\ell}(z)))\times\Aff^b,0\bigr),
\end{equation}
that is, there are local \'etale coordinates
\[
 z_0,\ldots,z_{\ell+\ell'-2},u_1,\ldots,u_b
\]
in which the ideal of $W^r_d(C)$ is
$I_{\ell-r}(\Cat_{\ell',\ell}(z))$ for every
$0\leq r\leq\ell-1$. The same conclusions hold therefore analytically.
\end{theorem}

Due to the homogeneity of the equations defining locally the Brill-Noether loci one has:

\begin{corollary}
For every $0\leq r\leq\ell-1$, the embedded germ
$(\Pic^d(C),W^r_d(C),L)$ is locally isomorphic for the \'etale and analytic topologies  to its scheme-theoretic tangent cone.
\end{corollary}

Theorem \ref{thm:main} is new even for hyperelliptic theta divisors. It answers positively our earlier \cite[Question~4.11]{BD} which we posed after noticing coincidences. The theorem has many consequences since there are many results known about the singularities of  catalecticant determinantal varieties. The state of the art on the latter at that time was summarized in \cite[Theorem~7.20]{BD}, and those results translate now into results for hyperelliptic Brill-Noether loci. That is, \cite[Proposition~4.14]{BD} holds unconditionally, we do not state it again here. In particular the results of Schnell-Yang \cite{SY} on hyperelliptic theta divisors
follow now effortlessly, a canonical log resolution is available for all hyperelliptic Brill-Noether loci, and one has the following  on multiplicities and the log canonical thresholds.

\begin{proposition}
Let $C$, $d$, $L$, $\ell$, $\ell'$ be as above, and $0\leq r\leq \ell-1$. Then
$$
 \operatorname{mult}_{L} W^r_d(C)
 =
 \binom{\ell'+r}{\ell-r-1}
$$
and

\vspace{-12pt}
$$
 \operatorname{lct}_{L}
 \bigl(\Pic^d(C),W^r_d(C)\bigr)
 =\left\{
 \begin{array}{ll}
1 & \text{ if }d=g-1\text{ and }r=0,\\
 1+\displaystyle\frac{\ell'+r-1}{\ell-r} & \text{ if }d<g-1\text{ or }r>0.
 \end{array}
 \right.
$$
\end{proposition}

The essential factor in \eqref{eq:all-r-etale-normal-form} is also the affine
cone  over the $(\ell-r-1)$-secant variety of the  rational normal
curve of degree $\ell+\ell'-2$, see \cite[Theorem~7.20(ii)]{BD}. 
Since the survey \cite{BD} there have been recent advances \cite{Brogan, CDOR} on secant varieties which now translate into new results on hyperelliptic Brill-Noether loci.

The results of Brogan
\cite{Brogan} on secant varieties give now a full description of the Milnor fibers and monodromy along the hyperelliptic theta divisors, which was one of the motivations behind \cite{Brogan}. We do not repeat those results here, and their translated version is immediate. Using the canonical log resolution one  verifies now that the Monodromy Conjecture, relating poles of the topological zeta function with monodromy eigenvalues, holds for hyperelliptic theta divisors.

Using  results of Chen--Dirks--Olano--Raychaudhury \cite{CDOR} on secant varieties we prove part~\textup{(i)} of the next theorem.  It has
several implications,  parts~\textup{(ii)--(iv)}, for the global
topology of the Brill--Noether loci.  Let $A$ denote the hyperelliptic
$g^1_2$, and $C^{(m)}$ the $m$-symmetric product.

\begin{theorem}
\label{thm:global-package}
Assume that $X=W_d^r(C)$ is nonempty, and retain the notation
\[
 m=d-2r,\qquad c=g-m,\qquad Z_s=W_d^{r+s}(C),
 \qquad 0\leq s\leq\left\lfloor\frac m2\right\rfloor.
\]
Then $m\geq0$ and the following statements hold.
\begin{enumerate}[label=\textup{(\roman*)},leftmargin=2.4em]
\item
The variety $X$ is a rational homology manifold of pure dimension $m$, and there is an identification of pure Hodge modules
\begin{equation}\label{eq:global-constant-IC}
 \QH_X[m]\simeq\IC^H_X.
\end{equation}

\item
The morphism
\[
 a_{d,r}:C^{(m)}\longrightarrow X,
 \qquad D\longmapsto A^r\otimes\OO_C(D),
\]
is a semismall resolution, and
\begin{equation}\label{eq:global-MHM-decomposition}
 Ra_{d,r*}\QH_{C^{(m)}}[m]
 \simeq
 \bigoplus_{s=0}^{\lfloor m/2\rfloor}
 \QH_{Z_s}[m-2s](-s)
\end{equation}
in $D^b\MHM(X)$.  Consequently, noncanonically as pure rational Hodge
structures,
\begin{equation}\label{eq:global-cohom-decomposition}
 H^k(C^{(m)},\QQ)
 \simeq
 \bigoplus_{s=0}^{\lfloor m/2\rfloor}
 H^{k-2s}(Z_s,\QQ)(-s),
\end{equation}
and in particular
when $m<2$, the next formula is to be read with its
second summand omitted,
\begin{equation}\label{eq:global-recursion}
 H^k(C^{(m)},\QQ)
 \simeq H^k(X,\QQ)\oplus
 H^{k-2}(C^{(m-2)},\QQ)(-1).
\end{equation}

\item
With the convention that binomial coefficients outside
$\{0,\ldots,g\}$ vanish,
\begin{equation}\label{eq:global-hodge}
 h^{p,q}(X)=
 \begin{cases}
  \binom gp\binom gq,&p+q\leq m,\\[4pt]
  \binom g{m-p}\binom g{m-q},&p+q\geq m,
 \end{cases}
\end{equation}
and hence
\begin{equation}\label{eq:global-betti}
 b_k(X)=
 \begin{cases}
  \binom{2g}{k},&0\leq k\leq m,\\[4pt]
  \binom{2g}{2m-k},&m\leq k\leq2m.
 \end{cases}
\end{equation}
Moreover, $X$ is formal over $\QQ$.

\item
Let $V=H^1(\Pic^d(C),\QQ)$ and let
$\theta\in\bigwedge^2V$ be the principal-polarization class.  Put
\[
 \mathcal I_{m,c}=
 \bigoplus_{k\geq m+1}
 \ker\!\left(
  \theta^c\wedge-:\bigwedge^kV\longrightarrow
  \bigwedge^{k+2c}V
 \right).
\]
Restriction induces an isomorphism of graded rational Hodge algebras
\begin{equation}\label{eq:global-ring}
 H^\bullet(X,\QQ)
 \simeq
 \frac{\bigwedge^\bullet V}{\mathcal I_{m,c}}.
\end{equation}
More precisely, restriction and Gysin give canonical isomorphisms
\begin{equation}\label{eq:global-Hodge-structures}
 H^k(X,\QQ)\simeq
 \begin{cases}
  \bigwedge^kV,&k\leq m,\\[3pt]
  \bigwedge^{k+2c}V(c),&k\geq m.
 \end{cases}
\end{equation}
At $k=m$ these agree through $(1/c!)\theta^c\wedge-$.  In particular,
$H^m(\Pic^d(C),\QQ)\xrightarrow{\sim}H^m(X,\QQ)$.
\end{enumerate}
\end{theorem}

\begin{remark}
Only 
\cref{thm:global-package}\textup{(i)} needs direct input from secant
varieties.  We use \cite[Corollary~L]{CDOR} to identify every proper
higher secant of a rational normal curve as a rational homology manifold and
\cite[Theorem~M]{CDOR} for its ordinary cohomology.  See also Brogan
\cite{Brogan}  for the even-degree secant
varieties relevant to hyperelliptic theta divisors.  Parts~\textup{(ii)--(iv)} use
directly  the rational Poincar\'e duality and
purity supplied by part~\textup{(i)}.

The method of part~\textup{(ii)} is the semismall extension of the argument
of Bressler--Brylinski \cite{BresslerBrylinski}.  In the nonhyperelliptic theta-divisor case they show that the
Abel--Jacobi map is small, so its direct image has only the top support.  

Equation (\ref{eq:global-hodge}) holds by elementary means if one only requires virtual Hodge numbers. Here we use the purity from part~\textup{(i)} for actual ordinary Hodge numbers. 

For a hyperelliptic theta divisor, the Betti numbers in
\eqref{eq:global-betti} appeared  in 
\cite[Theorem~3]{Nakayashiki}. However, we were kindly informed by A. Nakayashiki that the result \cite[Theorem~4]{Nakayashiki} (which is now covered by part~\textup{(i)} here) used for $b_k(X)$ with $m\le k\le 2m$, was not actually proven in \cite{BresslerBrylinski}, and therefore the computation in this range was left unsupported. 

The method for (iv) is the same as the one in \cite[\S4.2]{BresslerBrylinski} for the computation of the
intersection cohomology of a nonhyperelliptic theta divisor.  Their algebra structure is the one transported to intersection cohomology from the small resolution; it should not be confused with the intrinsic cup-product algebra on ordinary cohomology in part~\textup{(iv)}.

\end{remark}

\begin{remark}\label{remCru}
The local-model problem for $W^0_d(C)$ on a hyperelliptic curve was posed in \cite[Question~4.11]{BD}; \cite[Proposition~4.14(ii)]{BD} explains that a positive answer yields the corresponding local models for all $W^r_d(C)$. It was noted in \cite[Proposition 8.17]{BD} that the tangent cones to hyperelliptic Brill-Noether loci, given by the determinantal varieties of a matrix of linear forms associated to the Petri map of $L$ by a classical result of Kempf, are modelled by $\Cat_{\ell',\ell}(z)$. Then, adopting the strategy of the proof of the analog of Theorem \ref{thm:main} for generic curves from \cite{Bu}, the question to be answered in the hyperelliptic case was if the whole deformation theory with cohomology constraints of $L$ is catalecticant, that is, if an $\ell'\times\ell$ matrix of formal power series defining the germ of $W^0_d$ at $L$ satisfies the same equalities between the entries as $\Cat_{\ell',\ell}(z)$. Since the linear entries of this matrix of formal power series are linearly independent up to the catalecticant identities, a formal change of variables would then transform this matrix into $\Cat_{\ell',\ell}(z)$, proving Theorem \ref{thm:main}.
That the local scheme structure on $W^r_d$ can be defined by a matrix of the precise size $\ell'\times\ell$ is a phenomenon due to the approach with $L_\infty$ pairs in deformation theory and arises as a by-product of homotopy transfer. The usual definition of the scheme structure on $W^r_d$ from \cite{ACGH} is in general in terms of bigger size rectangular matrices of regular functions, see (\ref{eq:standard-ideal-r}).
\end{remark}

\begin{remark}[The use of AI] AI was essential.
We asked ChatGPT Pro  the question from Remark \ref{remCru}. The answer was elegant, based on the observation that the push-down of a Poincar\'e bundle to $\PP^1\times S$, where $S$ is a small neighborhood of $L$ in $\Pic^d(C)$, is a fixed extension of $\OO(l-1)$ by $\OO(-l'-1)$, see Proposition \ref{propCrucial}. The  $\ell'\times \ell$ matrix sought for arises  from the connecting map in cohomology. We noted it can be used to bypass \cite[Proposition 8.17]{BD} and  $L_\infty$ pairs. So we asked ChatGPT to rewrite the answer for all $W^r_d(C)$. Thus the proof of Theorem \ref{thm:main} is essentially due to AI. We then tasked Rethlas with editing the translation of some of the results of \cite{Brogan, CDOR}. More interestingly, it also suggested as consequences the deeper parts (ii)-(iv) of Theorem \ref{thm:global-package}. ChatGPT, Rethlas, and our own human effort were used for various check-ups and polishing of the exposition. We have verified all the AI generated content.
\end{remark}

This work was supported by KU Leuven Grant Methusalem
METH/21/03, FWO Grant G0B3123N, HIT Grant AUGA5610320926.  We thank VIASM, Hanoi, for its hospitality,
and P. Bressler and A. Nakayashiki for comments. 

\section{Hyperelliptic preliminaries}

We recall some classical facts about hyperelliptic curves. Fix the hyperelliptic morphism
$
 \pi:C\longrightarrow\PP^1
$
and put
$
 A=\pi^*\OO_{\PP^1}(1).
$
Let $\iota$ denote the hyperelliptic involution.

\begin{lemma}
There are isomorphisms
\begin{align}
 \pi_*\OO_C&\simeq
 \OO_{\PP^1}\oplus\OO_{\PP^1}(-g-1),
 \label{eq:pi-structure}\\
 K_C&\simeq A^{g-1}.
 \label{eq:canonical}
\end{align}
Moreover, for every $0\leq r\leq g$, pullback induces an isomorphism
\begin{equation}\label{eq:sections-A}
 H^0(\PP^1,\OO(r))\xrightarrow{\sim}H^0(C,A^r).
\end{equation}
\end{lemma}

\begin{proof}
The trace map for the finite flat double cover splits the unit inclusion
$\OO_{\PP^1}\hookrightarrow\pi_*\OO_C$ because the ground field has
characteristic zero.  Thus
\[
 \pi_*\OO_C\simeq\OO_{\PP^1}\oplus N
\]
for a line bundle $N$ on $\PP^1$.  Since finite pushforward preserves
cohomology,
\[
 \chi(\pi_*\OO_C)=\chi(\OO_C)=1-g.
\]
On $\PP^1$, a rank-two vector bundle $E$ satisfies
$\chi(E)=\deg E+2$.  Hence $\deg N=-g-1$, proving
\eqref{eq:pi-structure}.

By the projection formula,
\[
 H^0(C,A^r)
 \simeq H^0\bigl(\PP^1,
 \OO(r)\otimes\pi_*\OO_C\bigr)
 \simeq H^0(\PP^1,\OO(r))
 \oplus H^0(\PP^1,\OO(r-g-1)).
\]
The second summand vanishes for $0\leq r\leq g$, which proves
\eqref{eq:sections-A}.

It remains to identify the canonical bundle.  The line bundle $A^{g-1}$ has
degree $2g-2$, and \eqref{eq:sections-A} gives
$h^0(C,A^{g-1})=g$.  Riemann--Roch therefore yields
\[
 h^0\bigl(C,K_C\otimes A^{1-g}\bigr)=1.
\]
The line bundle $K_C\otimes A^{1-g}$ has degree zero.  A degree-zero line
bundle on a smooth projective curve with a nonzero section is trivial.
Thus $K_C\otimes A^{1-g}\simeq\OO_C$, proving
\eqref{eq:canonical}.
\end{proof}

Choose a base point on $C$ and a normalized Poincar\'e line bundle on
$C\times\Pic^d(C)$.  Let $S$ be an affine neighborhood of $L$ in
$\Pic^d(C)$, to be shrunk repeatedly, and let $\Lcal$ be the restricted
Poincar\'e line bundle on $C\times S$.  Write
\[
 p:C\times S\to S,\qquad
 q:\PP^1\times S\to S,\qquad
 \Pi=\pi\times\id_S.
\]
Since $\Pi$ is finite flat of degree two,
\[
 \Ecal=\Pi_*\Lcal
\]
is a rank-two vector bundle on $\PP^1\times S$, and
\begin{equation}\label{eq:finite-push}
 \Rp\Lcal\simeq\Rq\Ecal.
\end{equation}

\begin{lemma}\label{lem:central-splitting}
With the notation of \eqref{eq:basic-numerics},
\begin{equation}\label{eq:central-splitting}
 \pi_*L\simeq
 \OO_{\PP^1}(\ell-1)\oplus
 \OO_{\PP^1}(-\ell'-1).
\end{equation}
\end{lemma}

\begin{proof}
By the Grothendieck theorem, write
\[
 \pi_*L\simeq\OO_{\PP^1}(a)\oplus\OO_{\PP^1}(c),
 \qquad a\geq c.
\]
Because $\pi$ is finite,
\[
 h^i(\PP^1,\pi_*L)=h^i(C,L).
\]
Moreover,
\[
 \deg(\pi_*L)+2
 =\chi(\pi_*L)=\chi(L)=d+1-g,
\]
so
\[
 a+c=d-g-1\leq-2.
\]
The inequality $h^0(C,L)=\ell>0$ forces $a\geq0$.  The preceding sum then
forces $c\leq-2$.  Consequently only $\OO(a)$ contributes to $H^0$, and
only $\OO(c)$ contributes to $H^1$.  Hence
\[
 a+1=\ell,\qquad -c-1=\ell',
\]
which proves \eqref{eq:central-splitting}.
\end{proof}

\section{Hyperelliptic divisor theory and the Petri map}

For any effective divisor $D$ on $C$, let
\[
 \sigma_D\in H^0\bigl(C,\OO_C(D)\bigr)
\]
denote the canonical section: it is the image of $1$ under the natural
inclusion $\OO_C\hookrightarrow\OO_C(D)$ and satisfies
$\operatorname{div}(\sigma_D)=D$. The following is also classical.

\begin{lemma}\label{lem:fixed-divisor}
There is an effective divisor $B$ on $C$ such that
\begin{equation}\label{eq:L-decomposition}
 L\simeq A^{\ell-1}\otimes\OO_C(B),
 \qquad \deg B=b=d-2\ell+2.
\end{equation}
Consequently $b\geq0$.  The divisor $B$ contains no complete fiber of
$\pi$, and
\[
 K_C\otimes L^{-1}
 \simeq A^{\ell'-1}\otimes\OO_C(\iota B).
\]
If $V=H^0(\PP^1,\OO_{\PP^1}(1))$, multiplication by the canonical sections
$\sigma_B$ and $\sigma_{\iota B}$ gives isomorphisms
\begin{align}
 \Sym^{\ell-1}V&\xrightarrow{\sim}H^0(C,L),
 \label{eq:sections-L}\\
 \Sym^{\ell'-1}V&\xrightarrow{\sim}
 H^0(C,K_C\otimes L^{-1}).
 \label{eq:sections-dual}
\end{align}
\end{lemma}

\begin{proof}
Choose the inclusion of the nonnegative summand in
\eqref{eq:central-splitting},
\[
 \OO_{\PP^1}(\ell-1)\hookrightarrow\pi_*L.
\]
By adjunction it corresponds to a nonzero map of line bundles
\[
 A^{\ell-1}=\pi^*\OO_{\PP^1}(\ell-1)\longrightarrow L.
\]
It is injective, and its zero divisor is an effective divisor $B$.  Taking
degrees gives
\[
 \deg B=d-2(\ell-1)=d-2\ell+2=b,
\]
so $b\geq0$.  The induced map on global sections is the contribution of the
positive summand of $\pi_*L$; the negative summand has no global sections.
It is therefore an isomorphism, proving \eqref{eq:sections-L}.

If $B$ contained a complete hyperelliptic fiber, then
$L\simeq A^\ell\otimes\OO_C(B')$ for some effective divisor $B'$.  From
$b\geq0$ one has
\[
 \ell\leq\frac{d+2}{2}\leq\frac{g+1}{2}\leq g.
\]
Thus \eqref{eq:sections-A} gives $h^0(C,A^\ell)=\ell+1$.
Multiplication by the canonical section $\sigma_{B'}$ would inject
$H^0(C,A^\ell)$ into $H^0(C,L)$, contradicting
$h^0(C,L)=\ell$.  Hence $B$ contains no complete fiber.

Put $b=\deg B$.  For every point $x\in C$, the divisor
$x+\iota(x)$ is a fiber of $\pi$; at a ramification point this means the
doubled point.  Therefore
\begin{equation}\label{eq:B-iotaB}
 B+\iota B=\pi^*(\pi_*B),\qquad
 \OO_C(B+\iota B)\simeq A^b.
\end{equation}
Using \eqref{eq:canonical}, \eqref{eq:L-decomposition}, and
\eqref{eq:B-iotaB}, we obtain
\[
 \begin{aligned}
 K_C\otimes L^{-1}
 &\simeq A^{g-\ell}\otimes\OO_C(-B)\\
 &\simeq A^{g-\ell-b}\otimes\OO_C(\iota B)\\
 &\simeq A^{\ell'-1}\otimes\OO_C(\iota B),
 \end{aligned}
\]
because $g-\ell-b=\ell'-1$.  Since
$\ell'-1=g-\ell-b$ lies between $0$ and $g-1$, \eqref{eq:sections-A} gives
\[
 h^0(C,A^{\ell'-1})=\ell'.
\]
Riemann--Roch gives
$h^0(C,K_C\otimes L^{-1})=h^1(C,L)=\ell'$.  Multiplication by
$\sigma_{\iota B}$ is an injection between two vector spaces of dimension
$\ell'$, hence an isomorphism.  This proves
\eqref{eq:sections-dual}.
\end{proof}

Next result appeared as \cite[Proposition 8.17]{BD}, the strategy for which one of us learned at the time from G. Farkas.

\begin{proposition}\label{prop:petri}
The Petri multiplication map
\begin{equation}\label{eq:petri-map}
 \mu_L:H^0(C,L)\otimes H^0(C,K_C\otimes L^{-1})
 \longrightarrow H^0(C,K_C)
\end{equation}
has image of dimension $\ell+\ell'-1$.  In  monomial bases supplied by
\cref{lem:fixed-divisor}, its dual has image equal to the vector space of
$\ell'\times\ell$ rectangular catalecticant matrices.
\end{proposition}

\begin{proof}
The sections
\[
 \sigma_B\in H^0(C,\OO_C(B)),\qquad
 \sigma_{\iota B}\in H^0(C,\OO_C(\iota B))
\]
are the canonical sections defined at the beginning of this section; in
particular,
\[
 \operatorname{div}(\sigma_B)=B,\qquad
 \operatorname{div}(\sigma_{\iota B})=\iota B.
\]
Their product corresponds, under
$\OO_C(B+\iota B)\simeq A^b$, to a nonzero binary form
\[
 \eta\in H^0(\PP^1,\OO(b))=\Sym^bV.
\]
Under \eqref{eq:sections-L} and \eqref{eq:sections-dual}, the Petri map
factors as
\begin{equation}\label{eq:petri-factor}
 \Sym^{\ell-1}V\otimes\Sym^{\ell'-1}V
 \xrightarrow{\mathrm{mult}}
 \Sym^{\ell+\ell'-2}V
 \xrightarrow{\cdot\eta}
 \Sym^{g-1}V
 \xrightarrow{\sim}H^0(C,K_C).
\end{equation}
Here
\[
 b+\ell+\ell'-2=g-1.
\]
Multiplication of binary forms in the first arrow is surjective, and
multiplication by the nonzero binary form $\eta$ is injective.  Hence
\[
 \dim\operatorname{Im}\mu_L
 =\dim\Sym^{\ell+\ell'-2}V
 =\ell+\ell'-1.
\]
If $X,Y$ is a basis of $V$, the product of
$X^{\ell-1-j}Y^j$ and $X^{\ell'-1-i}Y^i$ depends only on $i+j$.
Dualizing \eqref{eq:petri-factor} therefore gives precisely the
rectangular catalecticant pattern displayed in
\eqref{eq:rectangular-cat}.
\end{proof}

\section{A fixed extension as the pushed-down Poincar\'e family}

Consider  the restricted
Poincar\'e line bundle $\Lcal$ on $C\times S$ and its push-down $\Ecal$ on $\PP^1\times S$.

\begin{proposition}\label{propCrucial}
After shrinking $S$ around $L$, there is an exact sequence on
$\PP^1\times S$
\begin{equation}\label{eq:fixed-extension}
 0\longrightarrow\OO(-\ell'-1)
 \longrightarrow\Ecal
 \longrightarrow\OO(\ell-1)
 \longrightarrow0,
\end{equation}
whose restriction to the central fiber is the split sequence associated with
\eqref{eq:central-splitting}.  Here and below $\OO(a)$ denotes the pullback
of $\OO_{\PP^1}(a)$ to $\PP^1\times S$.
\end{proposition}

\begin{proof}
Consider
\[
 \Fcal=\Ecal^\vee\otimes\OO(\ell-1).
\]
On the central fiber, \cref{lem:central-splitting} gives
\[
 \begin{aligned}
 \Fcal_L
 &\simeq
 \bigl(\OO(1-\ell)\oplus\OO(\ell'+1)\bigr)
 \otimes\OO(\ell-1)\\
 &\simeq\OO\oplus\OO(\ell+\ell').
 \end{aligned}
\]
Thus $H^1(\PP^1,\Fcal_L)=0$.  By semicontinuity and cohomology and base
change, after shrinking $S$ one has
$R^1q_*\Fcal=0$, the sheaf $q_*\Fcal$ is locally free, and its formation
commutes with base change.

The central projection
\[
 q_0:\pi_*L\twoheadrightarrow\OO_{\PP^1}(\ell-1)
\]
is a point of the fiber
\[
 (q_*\Fcal)_L
 =\Hom_{\PP^1}\bigl(\pi_*L,\OO(\ell-1)\bigr).
\]
After trivializing $q_*\Fcal$ on a neighborhood of $L$, extend $q_0$ to a
regular section of $q_*\Fcal$.  Using
\[
\Gamma(S,q_*\Fcal)
=
\Gamma(\PP^1\times S,\Fcal)
=
\Hom_{\PP^1\times S}
\bigl(\Ecal,\OO(\ell-1)\bigr),
\]
 this section is a morphism
\begin{equation}\label{eq:extended-quotient}
 \varphi:\Ecal\longrightarrow\OO(\ell-1)
\end{equation}
whose central restriction is $q_0$.

Let $Z\subset\PP^1\times S$ be the closed locus where $\varphi$ is not
surjective.  It is disjoint from the central fiber.  Since $q$ is proper,
$q(Z)$ is closed in $S$ and does not contain $L$.  Replacing $S$ by
$S\setminus q(Z)$ makes \eqref{eq:extended-quotient} surjective everywhere.
Its kernel $\Kcal$ is therefore a line bundle, and
\[
 0\longrightarrow\Kcal\longrightarrow\Ecal
 \xrightarrow{\varphi}\OO(\ell-1)\longrightarrow0
\]
is exact.

Every fiber $\Ecal_s$ has degree $d-g-1$, because
\[
 \deg(\Ecal_s)+2=\chi(\Ecal_s)=\chi(\Lcal_s)=d+1-g.
\]
It follows that
\[
 \deg\Kcal_s=(d-g-1)-(\ell-1)=-\ell'-1.
\]
Set
\[
 \Ncal=\Kcal\otimes\OO(\ell'+1).
\]
Each $\Ncal_s$ is a degree-zero line bundle on $\PP^1$, hence is trivial.
Cohomology and base change shows that $\Mcal=q_*\Ncal$ is a line bundle on
$S$, that $R^1q_*\Ncal=0$, and that the evaluation morphism
\[
 q^*\Mcal\longrightarrow\Ncal
\]
is an isomorphism on every fiber and therefore an isomorphism.  After
shrinking $S$, trivialize $\Mcal$.  Then
$\Kcal\simeq\OO(-\ell'-1)$, which proves \eqref{eq:fixed-extension}.
\end{proof}

\section{The extension class and the catalecticant boundary map}

Put
$
 W=H^1(\PP^1,\OO(-\ell-\ell')).
$
Since $\ell,\ell'\geq1$, one has
\begin{equation}\label{eq:W-dimension}
 \dim W=\ell+\ell'-1.
\end{equation}

\begin{lemma}
The sequence \eqref{eq:fixed-extension} is determined by a regular map
\begin{equation}\label{eq:xi-map}
 \xi:S\longrightarrow W
\end{equation}
with $\xi(L)=0$.
\end{lemma}

\begin{proof}
Because $S$ is affine and the two line bundles in
\eqref{eq:fixed-extension} are pulled back from $\PP^1$, Leray gives
\[
 \begin{aligned}
 \Ext^1_{\PP^1\times S}
 \bigl(\OO(\ell-1),\OO(-\ell'-1)\bigr)
 &\simeq H^1(\PP^1\times S,\OO(-\ell-\ell'))\\
 &\simeq H^0\bigl(S,R^1q_*\OO(-\ell-\ell')\bigr)\\
 &\simeq H^0(S,W\otimes\OO_S).
 \end{aligned}
\]
Thus the extension class is a regular section of the trivial vector bundle
$W\otimes\OO_S$, equivalently a regular map \eqref{eq:xi-map}.  The central
sequence is split, so $\xi(L)=0$.
\end{proof}

Choose a basis $X,Y$ of $V=H^0(\PP^1,\OO(1))$.  Serre duality gives
identifications
\begin{align*}
 H^0(\PP^1,\OO(\ell-1))&=\Sym^{\ell-1}V,
 \\
 H^1(\PP^1,\OO(-\ell'-1))&\simeq
 (\Sym^{\ell'-1}V)^\vee,
 \\
 W&\simeq(\Sym^{\ell+\ell'-2}V)^\vee.
\end{align*}
For $0\leq r\leq\ell+\ell'-2$, let $\epsilon_r\in W$ be dual to
$X^{\ell+\ell'-2-r}Y^r$.  Write
\begin{equation}\label{eq:xi-coordinates}
 \xi=\sum_{r=0}^{\ell+\ell'-2}\widetilde z_r\epsilon_r,
\end{equation}
where $\widetilde z_r\in\Gamma(S,\OO_S)$ and
$\widetilde z_r(L)=0$.  At this stage these are merely the coefficient
functions of the extension-class map $\xi$; they have not yet been
identified with the model coordinates $z_r$ introduced in
\eqref{eq:rectangular-cat}.
For brevity, write
\[
 \widetilde z=(\widetilde z_0,\ldots,
 \widetilde z_{\ell+\ell'-2}).
\]

Apply $\Rq$ to \eqref{eq:fixed-extension}.  Since
\[
 H^0(\PP^1,\OO(-\ell'-1))=0,
 \qquad
 H^1(\PP^1,\OO(\ell-1))=0,
\]
the result is represented by a two-term complex
\begin{equation}\label{eq:boundary-complex}
 H^0(\PP^1,\OO(\ell-1))\otimes\OO_S
 \xrightarrow{\ D\ }
 H^1(\PP^1,\OO(-\ell'-1))\otimes\OO_S,
\end{equation}
placed in cohomological degrees $0$ and $1$.  We identify $D$ next as
a connecting homomorphism and explain why this is the same as Yoneda
multiplication by $\xi$.

Set
\[
 \Acal=\OO(-\ell'-1),\qquad
 \Qcal=\OO(\ell-1).
\]
The short exact sequence \eqref{eq:fixed-extension} has extension class
\[
 \xi\in\Ext^1_{\PP^1\times S}(\Qcal,\Acal).
\]
Applying $\Rq$ gives the distinguished triangle
\begin{equation}\label{eq:boundary-triangle}
 \Rq\Acal\longrightarrow\Rq\Ecal\longrightarrow\Rq\Qcal
 \xrightarrow{\ \partial_\xi\ }\Rq\Acal[1].
\end{equation}
The displayed vanishings and cohomology and base change identify
\[
 \Rq\Qcal\simeq
 H^0(\PP^1,\OO(\ell-1))\otimes\OO_S
\]
in degree $0$, and
\[
 \Rq\Acal[1]\simeq
 H^1(\PP^1,\OO(-\ell'-1))\otimes\OO_S
\]
in degree $0$.  Choose the harmless sign in the identification of
$\Rq\Acal[1]$ with its degree-zero cohomology sheaf so that the differential
in the resulting two-term representative is the connecting morphism itself,
and define
\[
 D=\partial_\xi.
\]
With this convention, the triangle \eqref{eq:boundary-triangle} represents
$\Rq\Ecal$ by \eqref{eq:boundary-complex}.  Equivalently, the relative long exact
cohomology sequence is
\[
 0\longrightarrow q_*\Ecal\longrightarrow q_*\Qcal
 \xrightarrow{D}R^1q_*\Acal
 \longrightarrow R^1q_*\Ecal\longrightarrow0.
\]

The description by Yoneda multiplication is equally concrete.  Let
$U\subseteq S$ be affine and let
\[
 \sigma\in\Gamma(U,q_*\Qcal)
 =\Hom_{\PP^1\times U}
 (\OO_{\PP^1\times U},\Qcal|_{\PP^1\times U}).
\]
Pulling \eqref{eq:fixed-extension} back along $\sigma$  in the abelian category of sheaves of $\OO_{\PP^1\times U}$-modules gives an extension
\[
 0\longrightarrow\Acal|_{\PP^1\times U}
 \longrightarrow\Ecal_\sigma
 \longrightarrow\OO_{\PP^1\times U}
 \longrightarrow0.
\]
Its class is the Yoneda composite
\[
 \xi|_U\circ\sigma\in
 \Ext^1_{\PP^1\times U}
 (\OO_{\PP^1\times U},\Acal|_{\PP^1\times U}).
\]
Naturality of the connecting morphism identifies this pullback-extension
class with $\partial_\xi(\sigma)$.  Thus
\begin{equation}\label{eq:D-Yoneda}
 D(\sigma)=\xi|_U\circ\sigma.
\end{equation}
In particular, on the fiber over $s\in S$,
\[
 D_s(v)=\xi(s)\smile v
 \in H^1(\PP^1,\OO(-\ell'-1)),
 \qquad v\in H^0(\PP^1,\OO(\ell-1)),
\]
where cup product agrees with Yoneda composition under the standard
cohomology--Ext identifications.

\begin{proposition}[The all-orders catalecticant identity]
\label{prop:boundary-cat}
Let
\begin{align*}
 e_j&:=X^{\ell-1-j}Y^j\otimes 1
 \in H^0(\PP^1,\OO(\ell-1))\otimes\OO_S
 &&(0\leq j\leq\ell-1),\\
 f_i&:=X^{\ell'-1-i}Y^i\otimes 1
 \in \Sym^{\ell'-1}V\otimes\OO_S
 &&(0\leq i\leq\ell'-1),
\end{align*}
where $1$ is the unit section of $\OO_S$.  Use the $\OO_S$-bilinear
extension of the Serre-duality pairing to pair the target of
\eqref{eq:boundary-complex} with $\Sym^{\ell'-1}V\otimes\OO_S$.  Then
\[
 \langle D(e_j),f_i\rangle=\widetilde z_{i+j}.
\]
Consequently,
\[
 D=\Cat_{\ell',\ell}(\widetilde z)
\]
as an identity of matrices of regular functions on $S$.  In particular,
$D_{i+1,j}=D_{i,j+1}$ whenever both sides are defined.
\end{proposition}

\begin{proof}
By \eqref{eq:D-Yoneda}, $D(e_j)$ is Yoneda multiplication of the constant
section $e_j$ by the extension class $\xi$.  Pairing with the constant
section $f_i$ and using the $\OO_S$-bilinear extension of Serre duality gives
\[
 \langle D(e_j),f_i\rangle=\langle\xi,e_jf_i\rangle.
\]
Here the product on the right is formed using the $\OO_S$-bilinear
extension of the multiplication map on binary forms.  Hence
\[
 e_jf_i=X^{\ell+\ell'-2-(i+j)}Y^{i+j}\otimes1.
\]
The definition of the dual basis $\epsilon_r$ and
\eqref{eq:xi-coordinates} now give
$\langle\xi,e_jf_i\rangle=\widetilde z_{i+j}$.
\end{proof}

\section{The standard Brill--Noether filtration and the simultaneous normal form}

We finish in this section the proof of \cref{thm:main}.  We first compare the boundary complex with differential $D$ with the standard determinantal
construction of every $W^r_d(C)$ through $L$.

\subsection{The standard schemes and their higher Fitting ideals}
Choose an effective divisor $\Gamma$ on $C$ of degree $N$ so large that
$d+N>2g-2$, and put $\Gamma_S=\Gamma\times S$.  After shrinking $S$,
$H^1(C,L_s(\Gamma))=0$ for every $s\in S$.  Cohomology and base change gives
vector bundles
\begin{align*}
 \mathcal E^0_\Gamma&=p_*\Lcal(\Gamma_S),
 \\
 \mathcal E^1_\Gamma&=
 p_*\bigl(\Lcal(\Gamma_S)|_{\Gamma_S}\bigr).
\end{align*}
Restriction to $\Gamma_S$ defines
\[
 \rho_\Gamma:\mathcal E^0_\Gamma\longrightarrow
 \mathcal E^1_\Gamma.
\]
Put
\[
 c=g-d-1=\ell'-\ell,\qquad
 e=N-c=N+d+1-g.
\]
Riemann--Roch gives
\[
 \rank\mathcal E^0_\Gamma=e,\qquad
 \rank\mathcal E^1_\Gamma=N.
\]
Pushing forward
\[
 0\longrightarrow\Lcal\longrightarrow\Lcal(\Gamma_S)
 \longrightarrow\Lcal(\Gamma_S)|_{\Gamma_S}\longrightarrow0
\]
gives an exact sequence
\begin{equation}\label{eq:ACGH-exact}
 0\longrightarrow p_*\Lcal
 \longrightarrow\mathcal E^0_\Gamma
 \xrightarrow{\rho_\Gamma}\mathcal E^1_\Gamma
 \longrightarrow R^1p_*\Lcal\longrightarrow0.
\end{equation}
For $s\in S$, the kernel of $\rho_\Gamma(s)$ is $H^0(C,L_s)$.
Consequently, for every $0\leq r\leq\ell-1$, the condition
$h^0(C,L_s)\geq r+1$ is the rank condition
\[
 \rank\rho_\Gamma(s)\leq e-r-1.
\]
Following \cite[Chapter IV, \S3]{ACGH}, the standard scheme $W^r_d(C)\cap S$
is therefore defined by
\begin{equation}\label{eq:standard-ideal-r}
 I_{e-r}(\rho_\Gamma).
\end{equation}
This construction is independent of the sufficiently positive divisor
$\Gamma$ and of the auxiliary trivializations; see
\cite[Chapter IV, Remark~3.2]{ACGH}.

For a finite presentation
\[
 F_1\xrightarrow{\delta}F_0\longrightarrow M\longrightarrow0,
 \qquad \rank F_0=q,
\]
we use the convention
\[
 \Fitt_k(M)=I_{q-k}(\delta).
\]
It is well-known that Fitting ideals are independent of the chosen finite presentation.  From \eqref{eq:ACGH-exact},
\[
 \operatorname{coker}\rho_\Gamma\simeq R^1p_*\Lcal.
\]
On the other hand, the complex \eqref{eq:boundary-complex} represents
$\Rp\Lcal$ by \eqref{eq:finite-push} and
\eqref{eq:boundary-triangle}; hence
\[
 \operatorname{coker}D\simeq R^1p_*\Lcal.
\]
For $0\leq r\leq\ell-1$, the same Fitting ideal computed from these two
presentations is
\begin{equation}\label{eq:Fitting-comparison-r}
 \begin{aligned}
 I_{e-r}(\rho_\Gamma)
 &=\Fitt_{c+r}(R^1p_*\Lcal)\\
 &=I_{\ell-r}(D).
 \end{aligned}
\end{equation}
Indeed, the target ranks in the two presentations are $N$ and $\ell'$, and
\[
 N-(c+r)=e-r,\qquad
 \ell'-(c+r)=\ell-r.
\]

\begin{proposition}
\label{prop:standard-ideals}
For every $0\leq r\leq\ell-1$, the ideal of the standard Brill--Noether
scheme $W^r_d(C)\cap S$ is
$
 I_{\ell-r}(D).
$
In the extension functions of \eqref{eq:xi-coordinates}, this is
$
 I_{\ell-r}\bigl(\Cat_{\ell',\ell}(\widetilde z)\bigr).
$
All these equalities hold on the same neighborhood $S$ and for the same
matrix $D$.
\end{proposition}

\begin{proof}
The standard determinantal definition gives the ideal
$I_{e-r}(\rho_\Gamma)$.  Equality \eqref{eq:Fitting-comparison-r} identifies
it with $I_{\ell-r}(D)$, and \cref{prop:boundary-cat} identifies $D$ with
the rectangular catalecticant matrix.  Since neither comparison changes
$S$ or $D$ with $r$, the assertion is simultaneous.
\end{proof}


\subsection{First variation and the Petri map}

The central splitting fixes identifications
\begin{equation}\label{eq:central-cohomology}
 H^0(\PP^1,\OO(\ell-1))\simeq H^0(C,L),\qquad
 H^1(\PP^1,\OO(-\ell'-1))\simeq H^1(C,L).
\end{equation}

\begin{lemma}
\label{lem:first-variation}
Under \eqref{eq:central-cohomology} and
$T_L\Pic^d(C)=H^1(C,\OO_C)$, the differential of $D$ is
\begin{equation}\label{eq:cup-map}
 dD_L:H^1(C,\OO_C)\longrightarrow
 \Hom\bigl(H^0(C,L),H^1(C,L)\bigr),
 \qquad
 dD_L(\beta)(s)=\beta\smile s.
\end{equation}
Its Serre dual is the Petri multiplication map \eqref{eq:petri-map}.
\end{lemma}

\begin{proof}
Let $R_\varepsilon=\CC[\varepsilon]/(\varepsilon^2)$ and let
$\eta_\beta:\Spec R_\varepsilon\to S$ be the tangent arc corresponding to
$\beta\in H^1(C,\OO_C)$.  Pulling back the Poincar\'e family gives an
extension
\begin{equation}\label{eq:first-order-extension}
 0\longrightarrow L\xrightarrow{\varepsilon}\Lcal_\beta
 \longrightarrow L\longrightarrow0
\end{equation}
whose class in $\Ext^1_C(L,L)=H^1(C,\OO_C)$ is $\beta$.

Because \eqref{eq:boundary-complex} is a finite locally free representative
of $\Rp\Lcal$, derived base change to $\Spec R_\varepsilon$ identifies its
pullback with $\RGamma(C,\Lcal_\beta)$.  Under
\eqref{eq:central-cohomology}, the pulled-back complex is
\[
 \left[
 H^0(C,L)\otimes R_\varepsilon
 \xrightarrow{\eta_\beta^*D}
 H^1(C,L)\otimes R_\varepsilon
 \right].
\]
The central extension \eqref{eq:fixed-extension} is split, so $D(L)=0$ and
$
 \eta_\beta^*D=\varepsilon\,dD_L(\beta).
$
Multiplication by $\varepsilon$ and reduction modulo $\varepsilon$ give a
short exact sequence of complexes whose outer terms are
$[H^0(C,L)\xrightarrow{0}H^1(C,L)]$.  Its connecting map from degree zero
to degree one is $dD_L(\beta)$: a lift of $s\in H^0(C,L)$ has differential
$\varepsilon dD_L(\beta)(s)$, which corresponds to
$dD_L(\beta)(s)$ under
$\varepsilon H^1(C,L)\simeq H^1(C,L)$.

Under the derived-base-change identification, this is also the connecting
homomorphism of the line-bundle extension
\eqref{eq:first-order-extension}.  That connecting homomorphism is Yoneda,
equivalently cup, multiplication by the extension class $\beta$.  This
proves \eqref{eq:cup-map}.  If
$t\in H^0(C,K_C\otimes L^{-1})$, compatibility of cup product with Serre
duality gives
$
 \langle dD_L(\beta)(s),t\rangle=\langle\beta,st\rangle.
$
Thus the Serre dual of $dD_L$ is the Petri multiplication map.
\end{proof}

Let
\[
 \operatorname{cat}_{\ell',\ell}:W\longrightarrow
 \Hom\bigl(H^0(C,L),H^1(C,L)\bigr)
\]
be the linear map sending an extension class to its boundary homomorphism,
using \eqref{eq:central-cohomology}.  Its dual is multiplication
$
 \Sym^{\ell-1}V\otimes\Sym^{\ell'-1}V
 \longrightarrow\Sym^{\ell+\ell'-2}V,
$
which is surjective.  Hence $\operatorname{cat}_{\ell',\ell}$ is injective.
By construction,
\begin{equation}\label{eq:D-factor-xi}
 D=\operatorname{cat}_{\ell',\ell}\circ\xi.
\end{equation}

\begin{proposition}\label{prop:xi-submersion}
The differential
$
 d\xi_L:H^1(C,\OO_C)\longrightarrow W
$
is surjective.  Equivalently,
$
 d\widetilde z_{0,L}$, $\ldots$, $d\widetilde z_{\ell+\ell'-2,L}
$
are linearly independent.
\end{proposition}

\begin{proof}
By \cref{lem:first-variation,prop:petri},
$
 \rank(dD_L)=\rank(\mu_L)=\ell+\ell'-1.
$
Differentiating \eqref{eq:D-factor-xi} gives
$dD_L=\operatorname{cat}_{\ell',\ell}\circ d\xi_L$.
Since $\operatorname{cat}_{\ell',\ell}$ is injective, $d\xi_L$ has the same
rank as $dD_L$, namely $\ell+\ell'-1$.  This is $\dim W$ by
\eqref{eq:W-dimension}, so $d\xi_L$ is surjective.
\end{proof}

\begin{proof}[Proof of \cref{thm:main}]
The inequality $\ell'\geq\ell$ is \eqref{eq:ellprime-ge-ell}, while
$b\geq0$ was proved in \cref{lem:fixed-divisor}.  Direct calculation gives
$
 \ell+\ell'-1+b
 =\ell+(g-d+\ell-1)-1+(d-2\ell+2)=g.
$
By \cref{prop:xi-submersion}, the differentials
$
 d\widetilde z_{0,L},\ldots,d\widetilde z_{\ell+\ell'-2,L}
$
are linearly independent in the cotangent space of the smooth variety
$\Pic^d(C)$ at $L$.  After shrinking the affine neighborhood $S$, choose
regular functions $\widetilde u_1,\ldots,\widetilde u_b$ vanishing at $L$
whose differentials complete them to a basis. Let $z_0,\ldots,z_{\ell+\ell'-2},u_1,\ldots,u_b$ denote the standard
coordinate functions on
$\Aff^{\ell+\ell'-1}\times\Aff^b=\Aff^g$.  Define
\[
 \Psi:S\longrightarrow\Aff^g,\qquad
 s\longmapsto
 \bigl(\widetilde z_0(s),\ldots,\widetilde z_{\ell+\ell'-2}(s),
 \widetilde u_1(s),\ldots,\widetilde u_b(s)\bigr).
\]
The morphism $\Psi$ has invertible differential at $L$ and is therefore
\'etale at $L$.
Thus
$
 \Psi^*z_r=\widetilde z_r
 $, $
 \Psi^*u_j=\widetilde u_j.
$
Consequently, after using $\Psi$ to identify the source with the model
\'etale-locally, \cref{prop:boundary-cat} becomes the  identity
$
 D=\Cat_{\ell',\ell}(z),
$
with no dependence on the $u$-variables.  For every
$0\leq r\leq\ell-1$, \cref{prop:standard-ideals} says that the standard
scheme $W^r_d(C)$ is cut out on $S$ by $I_{\ell-r}(D)$.  Hence the same
map $\Psi$ identifies all the embedded germs simultaneously with
$
 V\!\left(I_{\ell-r}(\Cat_{\ell',\ell}(z))\right)\times\Aff^b.
$
\end{proof}

\section{Global topology and cohomology}

In this section we prove Theorem \ref{thm:global-package}. We use the setup and the notation as in Theorem \ref{thm:global-package}. Set
$Z_{\lfloor m/2\rfloor+1}=\varnothing$ and
$U_s=Z_s\setminus Z_{s+1}$ for
$0\leq s\leq\lfloor m/2\rfloor$.  Nonemptiness of $X$ and Clifford's
theorem imply $m\geq0$.
At a point $M\in U_s$ one has $h^0(C,M)=r+s+1$.  Applying
\cref{thm:main} at $M$ shows that the essential factor of the germ of $X$ is
the affine cone
\[
 X_s^{(N_s)}=C\!\left(\Sigma_s^{(N_s)}\right),
 \qquad N_s=c+2s-1,
\]
where $\Sigma_s^{(N_s)}$ is the $s$-secant variety of the degree-$N_s$
rational normal curve in $\PP^{N_s}$; the remaining smooth factor is
$\Aff^{m-2s}$.  In particular, the local dimension is $m$.  Notice that
$c\geq1$, hence $N_s\geq2s$. The following is well-known.

\begin{lemma}\label{lem:cone-rhm}
Let $Y\subset\PP^N$ be a projective rational homology manifold of complex
dimension $e$.  Assume that
\[
 H^{2j}(Y,\QQ)=\QQ(-j)\quad(0\leq j\leq e),
 \qquad H^{2j+1}(Y,\QQ)=0,
\]
and that the hyperplane class satisfies hard Lefschetz.  Then the affine cone
$C(Y)$ is a rational homology manifold.
\end{lemma}

\begin{proof}
Away from its vertex the cone is locally a product of $Y$ with a smooth
one-dimensional factor.  The link of the vertex is the unit-circle bundle of
$\OO_Y(-1)$.  Since the even cohomology groups are one-dimensional, hard
Lefschetz implies that cup product with the hyperplane class is an isomorphism
between consecutive nonzero groups.  The Gysin sequence therefore shows that
the link has the rational cohomology of $S^{2e+1}$.
\end{proof}

\begin{proof}[Proof of \cref{thm:global-package}\textup{(i)}]
For $s=0$ the essential factor is a point.  For $s=1$ it is the cone over a
rational normal curve; its link is the lens space obtained as the unit-circle
bundle of $\OO_{\PP^1}(-N_1)$, hence is a rational homology sphere.

Let $s\geq2$.  The line bundle $\OO_{\PP^1}(N_s)$ is
$(2s-1)$-very ample and $\Sigma_s^{(N_s)}\neq\PP^{N_s}$.  By
\cite[Corollary~L]{CDOR}, $\Sigma_s^{(N_s)}$ is a rational homology
manifold.  Specializing \cite[Theorem~M]{CDOR} to $C=\PP^1$, so that
$H^1(C,\QQ)=0$, gives
\[
 H^{2j}(\Sigma_s^{(N_s)},\QQ)=\QQ(-j)
 \quad(0\leq j\leq2s-1),
 \qquad H^{2j+1}(\Sigma_s^{(N_s)},\QQ)=0.
\]
Intersection-cohomology hard Lefschetz, identified with ordinary hard
Lefschetz by rational smoothness, verifies the remaining hypothesis of
\cref{lem:cone-rhm}.  Thus the affine cone $X_s^{(N_s)}$ \textcolor{red}{is} a rational homology
manifold.  The assertion for $X$ follows from Theorem \ref{thm:main} and
stability under products with smooth factors.
Finally, the canonical morphism
$\QH_X[m]\to\IC_X^H$ is an isomorphism on the underlying rational perverse
sheaves.  Faithfulness of the realization functor for mixed Hodge modules gives
\eqref{eq:global-constant-IC}.
\end{proof}


\begin{proposition}\label{prop:global-AJ}
The map $a_{d,r}$ is a resolution.  Each $U_s$ is smooth, and for
$M\in U_s$,
\[
 a_{d,r}^{-1}(M)=|M\otimes A^{-r}|\simeq\PP^s.
\]
Moreover, $a_{d,r}^{-1}(U_s)\to U_s$ is a Zariski locally trivial
$\PP^s$-bundle and
\[
 \dim Z_s=m-2s,
 \qquad \operatorname{codim}_X Z_s=2s,
\]
so $a_{d,r}$ is semismall; in a Whitney stratification of
the map subordinate to the pieces $U_s$, the dense stratum in each $U_s$ is
relevant (see \cite[Definitions~2.1.1 and~3.2.1]{deCataldoMigliorini} for
definitions).
\end{proposition}

\begin{proof} 
For $M\in U_s$, Lemma \ref{lem:fixed-divisor} gives
$M\simeq A^{r+s}\otimes\OO_C(B)$, where $B$ is effective of degree
$m-2s$ and contains no complete hyperelliptic fiber.  Hence
$h^0(M\otimes A^{-r})=s+1$, which proves the fiber formula and
surjectivity.  Constancy of $h^0$ and $h^1$ on $U_s$, followed by cohomology
and base change, identifies the family of complete linear systems with the
projectivization of a vector bundle.

The stratum $U_0$ is nonempty, so the map is birational.  Its source is smooth,
and its target is normal \cite[Theorem~7.20]{BD}; hence it is a resolution.  At a
point of $U_s$, the simultaneous local normal form for $Z_s=W_d^{r+s}(C)$ is
obtained by setting every essential catalecticant coordinate equal to zero.
Thus $U_s$ is smooth.  Applying the same construction to $Z_s$ gives
$\dim Z_s=m-2s$.  For $M\in U_s$ one therefore has
$
 2\dim a_{d,r}^{-1}(M)=2s=\operatorname{codim}_X U_s,
$
so equality holds on $U_s$ and hence on its dense stratum in
any such refinement.
\end{proof}

\begin{proof}[Proof of \cref{thm:global-package}\textup{(ii)}]
Choose a Whitney stratification of the proper algebraic map $a_{d,r}$
subordinate to the finite partition $X=\bigsqcup_s U_s$.  Such a refinement
exists, and \cref{prop:global-AJ} shows that the map is topologically locally
trivial over each of its strata.  For each $s$, the refinement has a unique
dense stratum $S_s\subseteq U_s$; its closure is $Z_s$ because $Z_s$ is
irreducible (it is the image of $C^{(m-2s)}$).  Along $S_s$ the fiber has
dimension $s$ and $\operatorname{codim}_X S_s=2s$, so $S_s$ is relevant.
Every other stratum contained in $U_s$ has strictly larger codimension while
the fiber dimension is still $s$, and is therefore not relevant.  Hence the
decomposition theorem
\cite[Corollaire~3]{SaitoHM},
\cite[Theorem~3.4.1]{deCataldoMigliorini} has exactly the supports $Z_s$.

Over $S_s$, the corresponding local system is the top cohomology of the
irreducible fiber, $H^{2s}(\PP^s,\QQ)=\QQ(-s)$.  It is constant because
$a_{d,r}^{-1}(U_s)\to U_s$ is the projectivization of a vector bundle and its
relative hyperplane class gives a global generator.
Thus
\[
 Ra_{d,r*}\QH_{C^{(m)}}[m]
 \simeq
 \bigoplus_s \IC^H_{Z_s}(-s).
\]
Part~\textup{(i)} identifies
$\IC^H_{Z_s}$ with $\QH_{Z_s}[m-2s]$, proving
\eqref{eq:global-MHM-decomposition}.  Hypercohomology gives
\eqref{eq:global-cohom-decomposition}.  If $m<2$, there
are no summands with $s\geq1$, so \eqref{eq:global-recursion} has the
stated boundary interpretation.  If $m\geq2$, reindexing the summands
with $s\geq1$ as the decomposition associated with
$C^{(m-2)}\to W_d^{r+1}(C)$ gives \eqref{eq:global-recursion}.
\end{proof}


\begin{proof}[Proof of \cref{thm:global-package}\textup{(iii)}]
Let $\mathcal H_Y(u,v)=\sum h^{p,q}(Y)u^pv^q$ for a projective variety with
pure cohomology.  The Hodge refinement of Macdonald's symmetric-product
formula \cite{Burillo} is
\[
 \sum_{n\geq0}\mathcal H_{C^{(n)}}(u,v)t^n
 =\frac{(1+ut)^g(1+vt)^g}{(1-t)(1-uvt)}.
\]
Using \eqref{eq:global-recursion} one easily obtains \eqref{eq:global-hodge} and  \eqref{eq:global-betti}.
Purity follows from
part~\textup{(i)}, and the purity-implies-formality theorem
\cite[Theorem~8.1]{CiriciHorel} proves the last assertion.
\end{proof}

\begin{proof}[Proof of \cref{thm:global-package}\textup{(iv)}]
Write $J_d=\Pic^d(C)$ and let
$
i:X\hookrightarrow J_d
$
be the inclusion.  Let
\[
u_m:C^{(m)}\longrightarrow\Pic^m(C),
\qquad
D\longmapsto\OO_C(D),
\]
be the Abel--Jacobi morphism.  It is shown in
\cite[\S4.2]{BresslerBrylinski} that
\[
u_m^*:H^k(\Pic^m(C),\QQ)
\longrightarrow H^k(C^{(m)},\QQ)
\]
is injective for every $k\leq m$.

Let
$
 \tau:J_d\to\Pic^m(C)$ be the isomorphism given by $M\longmapsto M\otimes A^{-r}$.
Then
$
 u_m=\tau\circ i\circ a_{d,r}.
$
Consequently,
$u_m^*=a_{d,r}^*\circ i^*\circ\tau^*$.  Since $u_m^*$ is injective and
$\tau^*$ is an isomorphism, $i^*$ is injective in every degree $k\leq m$.
By \eqref{eq:global-betti},
\[
 \dim H^k(X,\QQ)=\binom{2g}{k}
 =\dim H^k(J_d,\QQ)\qquad(k\leq m),
\]
so the restriction $$i^*: H^k(J_d,\QQ)\to H^k(X,\QQ)$$ is in fact an isomorphism in this range.

Poincar\'e's formula
\cite[Chapter~I, \S5]{ACGH} says
$
 (u_m)_*(1)=\theta_m^c/c!
 \in H^{2c}(\Pic^m(C),\QQ),
$
where $\theta_m$ is the principal-polarization class on $\Pic^m(C)$; the
left-hand side is the cohomology class of the image $W_m^0(C)$.  We have
$
 \tau(X)=W_m^0(C)
$
and $\tau$ transports the fundamental cycle of $X$ to the fundamental cycle
of $W_m^0(C)$.  Tensoring by the fixed line bundle $A^{-r}$ is a translation
between Picard components and therefore preserves the translation-invariant
principal-polarization class.  Thus $\tau^*\theta_m=\theta$, and Poincar\'e's
formula gives
\[
 [X]=\frac{\theta^c}{c!}\in H^{2c}(J_d,\QQ).
\]

By part~\textup{(i)}, the projective variety $X$ is a rational homology
manifold of complex dimension $m$.  It therefore has a fundamental class and
perfect rational Poincar\'e pairings.  If
$i_*^{\mathrm{BM}}$ denotes proper pushforward in Borel--Moore homology, the
cohomological Gysin map is defined by
\[
 i_*=\operatorname{PD}_{J_d}^{-1}\circ i_*^{\mathrm{BM}}
      \circ\operatorname{PD}_X:
 H^k(X,\QQ)\longrightarrow H^{k+2c}(J_d,\QQ)(c).
\]
It is a morphism of rational Hodge structures.  Moreover, $i_*$ is adjoint
to restriction, that is, for
$x\in H^k(X,\QQ)$ and $\beta\in H^{2m-k}(J_d,\QQ)$,
$
 \int_{J_d}i_*x\smile\beta
 =\int_X x\smile i^*\beta.
$
This  follows directly from the above definition and the
functoriality of the cap product.  We will use it below to identify $i_*$ with the dual of a complementary-degree restriction map.
We also have
\[
 i_*i^*(\alpha)=i_*(1\smile i^*\alpha)=i_*(1)\smile\alpha=
  [X]\smile\alpha
 =\frac1{c!}\theta^c\smile\alpha.
\]

We now determine the restriction $i^*$ in the remaining degrees.  Suppose first that
$m\leq k\leq2m$.  Under the perfect Poincar\'e pairings, the adjointness
identity identifies
\[
 i_*:H^k(X,\QQ)\longrightarrow H^{k+2c}(J_d,\QQ)(c)
\]
with the dual of
\[
 i^*:H^{2m-k}(J_d,\QQ)\longrightarrow H^{2m-k}(X,\QQ).
\]
Since $2m-k\leq m$, the latter map is an isomorphism by the low-degree
argument, and hence so is $i_*$.  For $k>2m$, both the source and the target of
$i_*$ vanish.  Thus $i_*$ is an isomorphism for every $k\geq m$.

Let $L=\theta\smile-$ be the Lefschetz operator on $J_d$.  Hard Lefschetz
implies that
\[
 L^c:H^k(J_d,\QQ)\longrightarrow H^{k+2c}(J_d,\QQ)
\]
is surjective for $k\geq m=g-c$. 
Let $y\in H^k(X,\QQ)$ with $k\geq m$.  Since $i_*$ is an isomorphism and
$L^c$ is surjective, choose $\alpha\in H^k(J_d,\QQ)$ such that
\[
 \frac1{c!}\theta^c\smile\alpha=i_*y.
\]
The formula for $i_*i^*$ gives $i_*i^*\alpha=i_*y$, and the injectivity of
$i_*$ yields $i^*\alpha=y$.  Thus the restriction $i^*$ is surjective in every degree
$k\geq m$.  The same argument shows, for such $k$, that
\[
 \ker i^*
 =\ker\!\left(
  \theta^c\smile-:H^k(J_d,\QQ)\longrightarrow
  H^{k+2c}(J_d,\QQ)
 \right):
\]
one inclusion follows by applying $i_*i^*$, and the reverse inclusion follows
from the injectivity of $i_*$.  In degree $m$, the map
$L^c:H^m(J_d,\QQ)\to H^{m+2c}(J_d,\QQ)$ is precisely the hard-Lefschetz
isomorphism, so this kernel is zero.

The  homomorphism $i^*$ is therefore surjective in all degrees, has no
kernel through degree $m$, and has in degree $k\geq m+1$ exactly the kernel
appearing in the definition of $\mathcal I_{m,c}$.  These kernels form a
homogeneous ideal because
$
 \theta^c\smile(\alpha\smile\beta)
 =(\theta^c\smile\alpha)\smile\beta.
$
This proves \eqref{eq:global-ring}.  Finally, the low-degree restriction
isomorphisms give the first line of \eqref{eq:global-Hodge-structures}, while
the high-degree Gysin isomorphisms give the second.  At degree $m$, their
composition is $(1/c!)\theta^c\smile-$ by the projection formula, which
explains the compatibility asserted in the theorem.
\end{proof}

\end{document}